\documentclass[a4paper,11pt]{article}
\usepackage{amsmath,amssymb,amsthm}
\usepackage[cp1251]{inputenc}
\usepackage[russian,english]{babel}
\inputencoding{cp1251}

\usepackage{graphicx}

\usepackage[pdfpagemode=UseOutlines,
 bookmarks=true,bookmarksopen=true,bookmarksnumbered=true,
 pdffitwindow=true,pdfpagelayout=SinglePage,pdfhighlight=/P,
 colorlinks=true,linkcolor=blue,citecolor=blue,urlcolor=blue,
 unicode=true]{hyperref}

\theoremstyle{plain}
\newtheorem{statement}{Statement}
\newtheorem{lemma}{Lemma}
\numberwithin{lemma}{section}

\numberwithin{corollary}{section}
\newtheorem{theorem}{Theorem}
\numberwithin{theorem}{section}

\numberwithin{proposition}{section}

\theoremstyle{definition}
\newtheorem{definition}{Definition}
\numberwithin{definition}{section}
\newtheorem{remark}{Remark}
\numberwithin{remark}{section}
\newtheorem{example}{Example}
\numberwithin{example}{section}

\title{The structure of the algebra of weak Jacobi forms for the root system $F_4$}
\author{D. Adler
\footnote 
{The author is supported by Laboratory of Mirror Symmetry NRU HSE, RF Government grant, number 14.641.31.0001.}}
\date{\today}

\begin{document}\maketitle

\begin{abstract}
We prove the polynomiality of the bigraded ring $J_{*,*}^{w, W}(F_4)$ of weak Jacobi forms for the root system $F_4$ which are invariant with respect to the corresponding Weyl group. This work is a continuation of the joint article with V.A. Gritsenko, where  the structure of algebras of the weak Jacobi forms related to the root systems of $D_n$ type for $2\leqslant n \leqslant 8$ was studied.

\smallskip


\end{abstract}

\section{Introduction}

The Chevalley theorem states that the algebra of polynomials which are invariant under the action of a finite group is polynomial if and only if this group is generated by (pseudo)reflections (see \cite{Ch}). The generalization of this result to the case of complex crystallographic Coxeter groups was studied in \cite{BS}, \cite{Lo1}, \cite{Lo2} and \cite{KP}. An analogue of the Chevalley theorem for weak Jacobi forms was obtained by K. Wirthm\"uller in \cite{Wirt} where the structure of the algebras was studied for all root systems, except $E_8$.
In the case of the root system $E_8$, as it has been proven by H. Wang in \cite{Wa}, the corresponding algebra of weak Jacobi forms is not polynomial.

The Wirthm\"uller's proof does not contain a direct construction for all generators of the corresponding algebras, but explicit formulas of them could be useful in applications. A one of such application is the computation of flat coordinates on some Frobenius manifolds (see \cite{Sa2}, \cite{Sa3}, \cite[\S 4]{Du}, \cite{Sat} and \cite{Ber1,Ber2}). For example, M. Bertola has considered the cases of the root systems $A_n$, $B_n$ and $G_2$ in papers \cite{Ber1} and \cite{Ber2}, and I. Satake has studied the case of the root system $E_6$ in \cite{Sat}.

The main goal of the present paper is to prove an analogue of the Chevalley theorem for the bigraded algebra of the weak Jacobi forms which are invariant under the action of the Weyl group for the lattice generated by the root system $F_4$. More precisely, we prove that this algebra is polynomial over the ring of modular forms. In  the joint paper with V. Gritsenko \cite{AG} we have proven the same result for the case of root systems $D_n$ with $2\leqslant n\leqslant 8$, and have explicitly constructed the generators of the corresponding algebras. In this paper, we make use of these generators for the case of $D_4$ and provide the direct construction of generators for the case $F_4$.

The paper is organized as follows. In \S 2 we give necessary definitions of root systems, Jacobi forms, and their examples. In \S 3 we introduce notations for  the coefficients of Fourier expansions of Jacobi forms which are invariant under the action of Weyl group. In \S 4 we state the main theorem of this paper and construct the generators of the algebra of weak Jacobi forms for $F_4$. In \S 5 we prove that constructed forms are algebraically independent over the ring of modular forms and generate the corresponding algebra of weak Jacobi forms.

The author is grateful to O. Schwarzman for the original formulation of the problem, helpful discussions of results and useful remarks to the paper, and to V. Gritsenko for supervision and very stimulating conversations.

\section{Definitions and constructions}

\subsection{Root systems}

In this paper we deal with the root systems of $D_n$ and $F_4$ type. Let us define them, following \cite{Bur}.

\begin{definition}\label{sysDn}
Let $\varepsilon_1,\ldots, \varepsilon_n$ be the standard orthonormal basis of $\mathbb{Z}^n$. Then $\alpha_1=\varepsilon_1 - \varepsilon_2, \ldots, \alpha_{n-1}=\varepsilon_{n-1}-\varepsilon_n$ and $\alpha_n=\varepsilon_{n-1}+\varepsilon_n$ form a basis of the root system $D_n$ and generate the lattice
$$D_n=\{(x_1,\ldots, x_n)\in \mathbb{Z}^n \,|\, \sum_{i=1}^n x_i \equiv 0 \mod{2}\}.$$
\end{definition}

By definition, this lattice is even, that is for any vector $\lambda\in D_n$ the square of its length $(\lambda, \lambda) \in 2\mathbb{Z}$.

Hereinafter, for simplicity, we use the same notation for the root system $D_n$ or $F_4$ and for the corresponding lattice.
It should be noted that it is possible to define the root system $D_n$ for $n=2$. In this case we obtain the reducible root system isomorphic to $A_1\oplus A_1$, where $A_1$ is the 1-dimensional root system with the basis $\varepsilon_1-\varepsilon_2\in \mathbb{R}^2$, and the corresponding lattice is
$$A_1=\{(x_1,  x_2)\in \mathbb{Z}^2 \,|\, x_1+x_2=0 \}.$$

The Weyl group for $D_n$ acts by permutations and by sign changes of even number of coordinates of vectors $\lambda=(x_1,\ldots, x_n)\in D_n$ in the standard basis. 

\begin{definition}\label{sysCn}
The vectors $\alpha_1=\varepsilon_1 - \varepsilon_2, \ldots, \alpha_{n-1}=\varepsilon_{n-1}-\varepsilon_n$ and $\alpha_n=2\varepsilon_n$ for $n\geqslant 3$ form a basis of the root system $C_n$. The lattice generated by the root system $C_n$ is the lattice $D_n$.
\end{definition}

For $n\neq 4$ the Weyl group $W(C_n)$ acts by permutations and by sign changes of any number of coordinates of vectors belonging to $D_n$. This group is also called the full orthogonal group $O(D_n)$ of the lattice $D_n$; it contains the Weyl group $W(D_n)$ as a subgroup of index 2.

For $n=4$ the Weyl group $W(C_n)$ acts the same way. Following \cite{AG}, we denote it $O\rq{}(D_4)$, because it is not equal to the full orthogonal group $O(D_4)$, and the exceptional root system $F_4$ appears.

\begin{definition}\label{sysF4}
The vectors
$$\tilde{\alpha}_1 =\varepsilon_2-\varepsilon_3, \quad \tilde{\alpha}_2 = \varepsilon_3-\varepsilon_4, \quad \tilde{\alpha}_3=\varepsilon_4, \quad \tilde{\alpha}_4=\frac{1}{2}(\varepsilon_1-\varepsilon_2-\varepsilon_3-\varepsilon_4)$$
form a basis of the root system $F_4$. 
The lattice generated by this root system is odd and contains the lattice $D_4$ as a sublattice of index 2.
\end{definition}

The Weyl group $W(F_4)$ corresponds to the full orthogonal group $O(D_4)$ and it is equal to the semidirect product of the symmetric group $S_3$ and the Weyl group $W(D_4)$. The group $W(F_4)$ is generated by reflections in $\pm \varepsilon_i$, $\pm \varepsilon_i\pm \varepsilon_j$ and $\frac{1}{2}(\pm \varepsilon_1\pm\varepsilon_2\pm\varepsilon_3\pm\varepsilon_4).$

\begin{remark}\label{A1n}
Let $L$ be a lattice with inner product $(\cdot, \cdot)$. Given an integer $m$, we denote by $L(m)$ the same lattice $L$, but with the inner product $m(\cdot, \cdot)$.
\end{remark}

\begin{example}
By definition, the lattice $A_1$ consists of vectors $(x, -x)$ with $x\in \mathbb{Z}$. The square of the length of any such vector is equal to $2x^2$. Hence, $A_1\simeq \mathbb{Z}(2)$. By Definition \ref{sysDn}, $D_n$ is a sublattice of $\mathbb{Z}^n$ of index 2. Therefore, 
$$
D_n(2)< \mathbb{Z}(2)^{\oplus n}\simeq A_1^{\oplus n}.
$$
\end{example}

\begin{example}\label{F4D4}
Let us consider the lattice $F_4(2)$. One can show by constructing an explicit isomorphism that this lattice is isomorphic to the lattice $D_4$ with standard inner product. 

As it was mentioned above, the Weyl group $W(F_4)$ is the orthogonal group of the lattice $D_4$. Thus, there are the following correspondences between the root systems and the pairs (lattice, group)
$$ D_4 \leftrightarrow (D_4, W(D_4)), \quad C_4 \leftrightarrow (D_4, O'(D_4)), \quad F_4 \leftrightarrow (D_4, O(D_4)).$$
\end{example}

\subsection{Jacobi forms}

For any positive definite lattice $L$ one can introduce the notion of Jacobi forms associated with this lattice. 

\begin{definition}\label{Jacobi}
Let $L$ be a positive definite lattice with the inner product $(\cdot\, , \cdot)$, let $\tau$ be a variable in the upper half-plane $\mathcal{H}$ and let $\mathfrak{z}=(z_1,\ldots, z_n)\in L\otimes\mathbb{C}$. Then \textit{a weak Jacobi form of weight $k\in\mathbb{Z}$ and index $m \in \mathbb{Z}$ for the lattice $L$} is
a holomorphic function $\varphi_{k, m}: \mathcal{H}\times(L\otimes\mathbb{C})\to\mathbb{C}$ which satisfies the functional equations
$$
\begin{aligned}
\varphi_{k, m}\left(\frac{a\tau+b}{c\tau+d},\frac{\mathfrak{z}}{c\tau+d}
\right)&=(c\tau+d)^ke^{\pi im\frac{c(\mathfrak{z},\mathfrak{z})}{c\tau
+d}}\varphi_{k, m}(\tau,\mathfrak{z})\;\text{for}\; 
\left(\begin{smallmatrix}
a & b\\
c & d
\end{smallmatrix}\right)\in SL_2(\mathbb{Z}),\\
\varphi_{k, m}(\tau, \mathfrak{z}+\lambda\tau+\mu)
&=e^{-2\pi im(\lambda,\,\mathfrak{z})-\pi im(\lambda,\,\lambda)\tau}\varphi_{k, m}(\tau,\mathfrak{z}) \;\text{for all}\; \lambda, \mu\in L
\end{aligned}
$$
and such that $\varphi_{k, m}(\tau,\mathfrak{z})$ has a Fourier expansion
$$\varphi_{k, m}(\tau,\mathfrak{z})=\sum_{l\in L^{\vee}}\sum_{n\geqslant 0} a(n, l)q^n \zeta^l.$$
\end{definition}

Hereinafter $L^{\vee}=\{m\in L\otimes \mathbb{Q} \,|\, \forall l\in L: (m,l)\in \mathbb{Z}\}$ is the dual lattice of $L$, $q=e^{2\pi i\tau}$, $\zeta^l=e^{2\pi i (\mathfrak{z},\, l)}$ for all $l$ from $L$ or $L^{\vee}$.

The first condition is called \textit{the modular equation}, the second one is called \textit{the quasiperiodicity equation}.
The set of all weak Jacobi forms has the natural structure of the bigraded ring 
$$J_{*,*}^{w}(L)=\bigoplus_{k,m}J_{k,m}^w(L).$$

\begin{remark}
If the function $\varphi_{k, m}(\tau,\mathfrak{z})$ satisfies the quasiperiodicity and the modular equations and has a Fourier expansion
$$\varphi_{k, m}(\tau,\mathfrak{z})=\sum_{l\in L^{\vee}}\sum_{n} a(n, l)q^n \zeta^l$$
where $a(n,l)\neq 0$ $\Rightarrow$ $2nm\geqslant (l, l)$, then it is called a \textit{holomorphic Jacobi form}. If, moreover, $\varphi_{k, m}(\tau,\mathfrak{z})$ satisfies the stronger condition $a(n,l)\neq 0$ $\Rightarrow$ $2nm > (l, l)$, then it is called a \textit{cusp Jacobi form}.
By these definitions,
$$J_{*,*}^c(L)\subset J_{*,*}(L)\subset J_{*,*}^w(L)$$
where $J_{*,*}^c(L)$ and $J_{*,*}(L)$ are the sets of cusp and holomorhic Jacobi forms, respectively. 
Let us note here that the holomorphic Jacobi forms are often called just \text{Jacobi forms}, after \cite{EZ}. However, to avoid confusion, in this paper we use the name Jacobi form only when we consider a form of general type, that is when it does not matter whether this form is weak, or holomorhic or cusp. Also it should be noted that the main theorem \ref{MainTheorem} of this paper does not apply to the cases of holomorphic and cusp Jacobi forms.
\end{remark}

\begin{remark}\label{index}
Let $\varphi(\tau,\mathfrak{z})$ be a Jacobi form of weight $k$ and index $m$ for the lattice $L$. Then, by definitions, $\varphi(\tau,\mathfrak{z})$ is also a Jacobi form of the same type (weak, holomorhic or cusp), of weight $k$ and of index 1 for the lattice $L(m)$ (see Remark \ref{A1n}).
\end{remark}

\begin{lemma}\label{oddLattice}
Let $L$ be an odd positive definite lattice. Then the index of any Jacobi form for this lattice is even.
\end{lemma}
\begin{proof}
Let us consider an arbitrary Jacobi form $\varphi_{k, m}$ of weight $k$ and index $m$ for $L$.
After substitution of $\tau +1$ instead of $\tau$ in the quasiperiodicity equation we obtain 
$$\varphi_{k, m}(\tau+1, \mathfrak{z}+\lambda(\tau+1)+\mu)
=e^{-2\pi im(\lambda,\,\mathfrak{z})-\pi im(\lambda,\,\lambda)(\tau+1)}\varphi_{k, m}(\tau+1,\mathfrak{z}).$$
The modular equation and the periodicity under shifts $z\mapsto z+\lambda$ give
$$\varphi_{k, m}(\tau, \mathfrak{z}+\lambda\tau+\mu)=\varphi_{k, m}(\tau, \mathfrak{z}+\lambda\tau+(\lambda+\mu))=$$$$
=e^{-2\pi im(\lambda,\,\mathfrak{z})-\pi im(\lambda,\,\lambda)(\tau+1)}\varphi_{k, m}(\tau,\mathfrak{z})=$$
$$=e^{-\pi im(\lambda, \lambda)}e^{-2\pi im(\lambda,\,\mathfrak{z})-\pi im(\lambda,\,\lambda)\tau}\varphi_{k, m}(\tau,\mathfrak{z}).$$
Comparing this equation with the quasiperiodicity equation we conclude that $m(\lambda, \lambda)\in 2\mathbb{Z}$ for all $\lambda \in L$.
Therefore, all Jacobi forms for the odd lattice $L$ have even indices. Or, equivalently, as it follows from Remark \ref{index}, it is possible to consider Jacobi forms of an arbitrary integer weight, but for the even lattice $L(2)$.
\end{proof}

\begin{definition}
Let $G$ be a subgroup of the full orthogonal group $O(L)$ of the positive definite lattice $L$. The Jacobi form $\varphi_{k, m}(\tau,\mathfrak{z})$ for this lattice is called a $G$-invariant Jacobi form if for all $g\in G$ 
$$ \varphi_{k, m}(\tau, g(\mathfrak{z}))=\varphi_{k, m}(\tau, \mathfrak{z}).$$
\end{definition}

In this paper we consider weak $G$-invariant Jacobi forms for positive definite lattices. We denote the bigraded algebra of the corresponding forms
$$J_{*,*}^{w, G}(L)=\bigoplus_{k,m}J_{k,m}^{w, G}(L).$$
Our main goal is to prove that the algebra of weak $W(F_4)$-invariant Jacobi forms is polynomial. As it was mentioned above the lattice $F_4$ is odd. Hence, by Lemma \ref{oddLattice}, we consider Jacobi forms for the lattice $F_4(2)$. However, by Example \ref{F4D4}, 
$$J_{*, *}^{w, W}(F_4(2))\simeq J_{*, *}^{w, O}(D_4).$$
Therefore, we study the structure of the algebra of weak $O(D_4)$-invariant Jacobi forms.

Any $O(D_4)$-invariant Jacobi form is $W(D_4)$-invariant because $W(D_4)$ is a subgroup of $O(D_4)$. So, the generators of  $J_{*, *}^{w, O}(D_4)$ are $W(F_4)/W(D_4)\simeq S_3$-invariant polynomials in generators of the algebra of weak $W(D_4)$-invariant Jacobi forms (the other conditions of Definition \ref{Jacobi} are satisfied automatically).

\subsection{Examples of Jacobi forms}

In this section we give the main examples of Jacobi forms and, following \cite{AG}, generators of $J_{*, *}^{w, W}(D_4)$.

\begin{example}\label{Theta}
One of the most important functions in the theory of Jacobi forms is the odd Jacobi theta-function (see e.g. \cite{M}):
$$
\vartheta(\tau, z) =-i\vartheta_{11}(\tau,z)= q^{\frac{1}{8}} \sum_{n\in \mathbb{Z}} (-1)^n 
q^{\frac{n(n+1)}{2}} \zeta^{n+\frac{1}{2}} = $$
$$=
-q^{\frac{1}{8}}\zeta^{-\frac{1}{2}}
\prod_{n = 1}^{\infty} (1-q^{n-1} \zeta)(1-q^n\zeta^{-1})(1-q^n).
$$
Formally, $\vartheta(\tau,z)$ is not a Jacobi form (if we consider forms of integer weight and without characters), because
\begin{gather*}
   \label{eq11}
   \vartheta(\tau,z+\lambda\tau+\mu)=(-1)^{\lambda+\mu}e^{-\pi i(\lambda^2\tau+2\lambda z)}
   \vartheta(\tau,z),\quad
   \lambda, \mu \in \mathbb{Z},
   \\
   \label{eq12}
\vartheta\left(\frac{-1}{\tau},\frac{z}{\tau}\right)
=-i\sqrt{i\tau}
   e^{\pi i\frac{z^2}{\tau}}\vartheta(\tau,z).
\end{gather*}
However, combining it with the Dedekind $\eta$-function 
$$\eta(\tau)=q^{\frac{1}{24}}\prod_{n\geqslant 1}(1-q^n)$$
one can construct many examples of Jacobi forms. 
For example, the weak Jacobi form of weight $-2$ and index 1 from \cite{EZ} can be represented as
$$\phi_{-2,1}(\tau, z)=\frac{\vartheta(\tau, z)\vartheta(\tau, -z)}{\eta^6(\tau)} 
=(\zeta -2 +\zeta^{-1}) +q\cdot(\ldots) \in J_{-2,1}^{w, W}(A_1),$$
because $A_1\simeq\mathbb{Z}(2)$, and the Weyl group acts on $A_1$ by the sign changes (see Remark \ref{A1n}).
\end{example}

\begin{definition}\label{DifOp}
Another method to construct Jacobi forms is a modular differential operator. More precisely, it is the operator $H_k$ acting on Jacobi forms of weight $k$ and index $m$ for a lattice $L$ of rank $n_0$ with the inner product $(\cdot, \cdot)$ by
$$
H^{(L)}_k(\varphi_{k,m})(\tau, \mathfrak{z}) = $$
$$
=\frac{1}{2\pi i} \frac{\partial \varphi_{k,m}}{\partial \tau}(\tau, 
\mathfrak{z})+\frac{1}{8\pi^2 m}\left(\frac{\partial}{\partial 
\mathfrak{z}}, \frac{\partial}{\partial \mathfrak{z}}\right)
\varphi_{k,m}(\tau, \mathfrak{z}) +  (2k-n_0)G_2(\tau)\varphi_{k,m}(\tau, 
\mathfrak{z})=$$
$$= \sum_{n=0}^{\infty}\sum_{l \in L^{\vee}} \left(n-\frac{1}{2m}(l, l)
\right) a(n, l) q^n \zeta^{l}+  (2k-n_0)G_2(\tau)\varphi_{k,m}(\tau, 
\mathfrak{z}),$$
where $G_2(\tau)=-\frac{1}{24}+\sum_{n\geqslant 1} \sigma(n) q^n$ is the quasimodular Eisenstein series of weight 2 and $$\sigma_k(n)=\sum_{d\vert n}d^k.$$
\end{definition}

One can check that this operator transforms weak (holomorphic, cusp) Jacobi forms of weight $k$ and index $m$ into weak (holomorphic, cusp) Jacobi forms of weight $k+2$ and the same index $m$. For details see \cite{GrJ}.

\begin{example}\label{weight0}
Using this operator one can construct the weak Jacobi form of weight 0 and index 1 from \cite{EZ}:
$$\phi_{0,1}(\tau, z)=H_{-2}^{A_1}(\phi_{-2,1}(\tau, z))
=(\zeta +10 +\zeta^{-1}) +q\cdot(\ldots) \in J_{0,1}^{w, W}(A_1).$$
As it was proven in \cite{EZ},
$$J^{w, W}_{*,*}(A_1)=M_*[\phi_{0,1}, \phi_{-2,1}],$$
where $$M_*=\bigoplus_{k\geqslant 0} M_{2k}(SL_2(\mathbb{Z}))=\mathbb{C}[E_4, E_6]$$
is a ring of modular forms. This ring is generated by two Eisenstein series 
\begin{equation}\label{Eis}
E_4(\tau)=1+240\sum_{n\geqslant 1}\sigma_3(n)q^n,\quad 
E_6(\tau)=1-504\sum_{n\geqslant 1}\sigma_5(n)q^n.
\end{equation}

\end{example}

\begin{example}\label{index2}
Let $\varphi_1(\tau, \mathfrak{z}_1)$ be a weak Jacobi form of weight $k_1$ and index $m$ for the lattice $L_1$, and let $\varphi_2(\tau, \mathfrak{z}_2)$ be a weak Jacobi form of weight $k_2$ and the same index $m$ for the lattice $L_2$. Then for the lattice $L=L_1\oplus L_2$ (here we consider the orthogonal sum of the lattices) the form
$$\varphi(\tau, \mathfrak{z})=\varphi_1(\tau, \mathfrak{z}_1)\varphi_2(\tau, \mathfrak{z}_2)$$
is a weak Jacobi form of weight $k_1+k_2$ and index $m$. The averaging over a subgroup $G$ of the $O(L)$ group gives $G$-invariant Jacobi forms. 

For example, if $L=A_1^{\oplus n}$, then for any $k$ the form
\begin{multline*}
\varphi_{-2k,1}^{L}=\frac{1}{k!(n-k)!}\sum_{\sigma\in S_n} \phi_{-2, 1}(\tau, z_{\sigma(1)})\ldots\phi_{-2, 1}(\tau, z_{\sigma(k)}) \times \\
\times\phi_{0, 1}(\tau, z_{\sigma(k+1)})\ldots\phi_{0, 1}(\tau, z_{\sigma(n)})
\end{multline*}
is a weak Jacobi form of weight $-2k$ and index 1 for $(A_1)^{\oplus n}$. Here $\phi_{-2, 1}$ and $\phi_{0, 1}$ are the same as in Examples \ref{Theta} and \ref{weight0}.
By Remark \ref{A1n}, all such forms $\varphi_{-2k,1}^{L}$ are weak $O(D_n)$-invariant (or $O\rq{}(D_4)$-invariant, if $n=4$) Jacobi forms of weight $-2k$ and index 2.

The case of holomorphic or cusp Jacobi forms is practically the same, but we do not consider it in this paper.
\end{example}

\begin{example}\label{ThetaD4}
In \cite[Example 1.8]{CG}, the functions
$$\vartheta_{D_4}(\tau, \mathfrak{z})=\vartheta(\tau, z_1)\cdot\ldots\cdot\vartheta(\tau, z_4),$$
$$\vartheta_{D_4}^{(2)}(\tau, \mathfrak{z})=\vartheta\left(\tau,\frac{-z_1+z_2+z_3+z_4}{2}\right)\vartheta\left(\tau, \frac{z_1-z_2+z_3+z_4}{2}\right)\times$$ $$\times\vartheta\left(\tau, \frac{z_1+z_2-z_3+z_4}{2}\right)\vartheta\left(\tau, \frac{z_1+z_2+z_3-z_4}{2}\right),$$
$$\vartheta_{D_4}^{(3)}(\tau, \mathfrak{z})=\vartheta\left(\tau,\frac{z_1+z_2+z_3+z_4}{2}\right)\vartheta\left(\tau, \frac{z_1-z_2-z_3+z_4}{2}\right)\times$$ $$\times\vartheta\left(\tau, \frac{z_1+z_2-z_3-z_4}{2}\right)\vartheta\left(\tau, \frac{z_1-z_2+z_3-z_4}{2}\right)$$
have been introduced for the lattice $D_4$.
These functions are not Jacobi forms in the sense of Definition \ref{Jacobi}, because they transform with some multiplicative characters under the action of $SL_2(\mathbb{Z})$ and the shifts by elements of the lattice $D_4$. However, like in  Example \ref{Theta}, the use of suitable power of the Dedekind $\eta$-function makes possible to get rid of these characters and to obtain weak $W(D_4)$-invariant Jacobi forms of weight $-4$ and index 1:
$$\omega_{-4, 1}(\tau, \mathfrak{z}) =\frac{\vartheta_{D_4}(\tau, \mathfrak{z})}{\eta^{12}(\tau)}=\frac{\vartheta(\tau, z_1)\cdot\ldots\cdot\vartheta(\tau, z_4)}{\eta^{12}(\tau)},$$
$$\varphi_{-4,1}^{(2)}(\tau, \mathfrak{z})=\frac{\vartheta^{(2)}(\tau, \mathfrak{z})}{\eta^{12}(\tau)} ~~ \text{and} ~~ \varphi_{-4,1}^{(3)}(\tau, \mathfrak{z})=\frac{\vartheta^{(3)}(\tau, \mathfrak{z})}{\eta^{12}(\tau)}.$$
The form $\omega_{-4, 1}$ vanishes at each $z_i\in\mathbb{Z}\oplus\mathbb{Z}\tau$. Also this form is anti-invariant under sign changes of odd number of $\mathfrak{z}$-coordinates. The other two forms are neither invariant nor anti-invariant, but one can check that $\omega_{-4,1} = \varphi_{-4,1}^{(2)}+\varphi_{-4,1}^{(3)}$, and that  $\varphi_{-4,1}=\varphi_{-4,1}^{(2)} - \varphi_{-4,1}^{(3)}$ is invariant under such transformations.
\end{example}

\begin{remark}
By analogy with construction of $\omega_{-4, 1}(\tau, \mathfrak{z})$, for any lattice $D_n$ the form
$$\omega_{-n, 1}(\tau, \mathfrak{z}) =\frac{\vartheta(\tau, z_1)\cdot\ldots\cdot\vartheta(\tau, z_n)}{\eta^{3n}(\tau)}$$
is a weak Jacobi form of weight $-n$ and index $1$, which is anti-invariant under sign changes of odd number of $\mathfrak{z}$-coordinates, and its divisor consists of $z_i\in\mathbb{Z}\oplus\mathbb{Z}\tau$.
\end{remark}

In \cite{AG}, the author of this paper and V. Gritsenko proved that set of all $W(D_n)$-invariant weak Jacobi forms for the lattice $D_n$ with $3\leqslant n\leqslant 8$ has the structure of a polynomial algebra with $n+1$ generators over the ring of modular forms. Moreover, these generators were constructed in an explicit way. Here we need this result only in the case $n=4$. 

\begin{theorem}\label{TheoremD4}
The bigraded algebra of weak $W(D_4)$-invariant Jacobi forms for the lattice $D_4$ is polynomial
$$J^{w, W}_{*,*}(D_4)=M_*[\varphi_{0,1}, \varphi_{-2,1}, \varphi_{-4,1}, \varphi_{-6,2}, \omega_{-4,1}].$$
The form $\omega_{-4,1}$ is anti-invariant under sign changes of odd number of $\mathfrak{z}$-coordinates, and the other forms are invariant. Moreover,  
$$J^{w, W}_{*,*}(D_3)=M_*\left[\varphi_{0,1}\bigg|_{z_4=0}, \varphi_{-2,1}\bigg|_{z_4=0}, \varphi_{-4,1}\bigg|_{z_4=0}, \omega_{-3,1}\right].$$
\end{theorem}

The forms $\varphi_{-4,1}$ and $\omega_{-4,1}$ here are exactly the same as in Example \ref{ThetaD4}. Construction of $\varphi_{-6,2}$ is described in Example \ref{index2}. Let us note that for the lattice $D_4$ the images of both $\varphi_{-4,1}$ and $\omega_{-4,1}$ under the differential operator $H_{-4}$ are identically equal to zero (by direct computation). However, for $D_n$ with $n\neq 4$
$$ H_{-4}(\varphi_{-4,1})=\varphi_{-2,1}\neq 0,$$
while $H_{-4}(\omega_{-4,1})$ is still equal to zero; the form $\varphi_{-2,1}\neq 0$ for $D_4$ can be obtained from $\varphi_{-2,1}$ for $D_n$ with $n>4$ by restriction on $D_4$ (for details see \cite{AG}). The last form $\varphi_{0,1}$ can be obtained, for example, as 
$$\varphi_{0,1}=4H_{-2}(\varphi_{-2,1})+\frac{1}{3}E_4\varphi_{-4,1}.$$
We add the last summand here because of some technical details (see \cite{AG}).

\begin{remark}\label{change}
In further constructions we also use representation of $\varphi_{-4,1}$ and $\omega_{-4,1}$ in terms of $\varphi_{-4,1}^{(2)}$ and $\varphi_{-4,1}^{(3)}$. Let us note that the transformation between these two pairs of forms is invertible. So, the form $\varphi \in J^{w, W}_{*,*}(D_4)$ can also be uniquely represented as a polynomial in $\varphi_{0,1}$, $\varphi_{-2,1}$, $\varphi_{-4,1}^{(2)}$, $\varphi_{-6,2}$ and $\varphi_{-4,1}^{(3)}$.
\end{remark}

\section{Invariant Fourier coefficients}

As it was mentioned in Definition \ref{Jacobi}, a weak Jacobi form has a Fourier expansion of the following type:
$$\varphi(\tau,\mathfrak{z})=\sum_{l\in L^{\vee}}\sum_{n\geqslant 0} a(n, l)q^n \zeta^l.$$
Let $L$ be a lattice generated by a root system. Its Weyl group acts on elements of $L$ and $L^{\vee}$. So, if any two vectors $l_1$ and $l_2$ belong to the same orbit under the action of the Weyl group, then any $W(L)$-invariant Jacobi form satisfies $a(n, l_1) = a(n, l_2)$.
Consequently, each $q^n$-coefficient can be represented as a sum of $W(L)$-invariant polynomials in $\zeta^l$. For the lattice $D_4$ we introduce the following polynomials: 
$$Q_0=1 \quad \text{corresponding to the orbit of 0},$$
$$P_1=\sum_{j=1}^4 (\zeta_j+\zeta_j^{-1})  \quad \text{corresponding to the orbits of} \; \varepsilon_1 \; \text{and} \; -\varepsilon_1,$$
$$P_{\frac{4}{4}}^{+}=\sum_{\{\text{even number of +}\}}\zeta_1^{\pm\frac{1}{2}}\zeta_2^{\pm\frac{1}{2}}\zeta_3^{\pm\frac{1}{2}}\zeta_4^{\pm\frac{1}{2}}  \quad  \text{for the orbit of} \quad \frac{1}{2}(\varepsilon_1+\varepsilon_2+\varepsilon_3+\varepsilon_4),$$
$$P_{\frac{4}{4}}^{-}=\sum_{\{\text{odd number of +}\}}\zeta_1^{\pm\frac{1}{2}}\zeta_2^{\pm\frac{1}{2}}\zeta_3^{\pm\frac{1}{2}}\zeta_4^{\pm\frac{1}{2}}  \quad \text{for the orbit of} \quad \frac{1}{2}(\varepsilon_1+\varepsilon_2+\varepsilon_3-\varepsilon_4),$$
$$P_{\frac{4}{4}}=P_{\frac{4}{4}}^{+}+P_{\frac{4}{4}}^{-},$$
$$Q_2=\sum_{1\leqslant i<j\leqslant 4} \zeta_i^{\pm1}\zeta_j^{\pm1}  \quad \text{for the orbit of roots} \quad \pm\varepsilon_1\pm\varepsilon_2,$$
$$P_3 = \sum_{1\leqslant i<j<k \leqslant 4} \zeta_i^{\pm1}\zeta_j^{\pm1}\zeta_k^{\pm 1}  \quad\text{for the orbit of elements of the type} \; \pm \varepsilon_1 \pm \varepsilon_2 \pm \varepsilon_3,$$
$$Q_4^1 = \sum_{1\leqslant i<j<k<l \leqslant 4} \zeta_i^{\pm1}\zeta_j^{\pm1}\zeta_k^{\pm1}\zeta_l^{\pm1} \; \text{for the orbit of elements of the type} \; \sum_{i=1}^4 \pm \varepsilon_i,$$
$$Q_4^2=\sum_{j=1}^4 (\zeta^2_j+\zeta_j^{-2}) \quad  \text{corresponding to the orbits of} \quad 2\varepsilon_1 \; \text{and} \; -2\varepsilon_1,$$
$$Q_4=Q_4^1+Q_4^2,$$
$$P_{\frac{12}{4}}  \quad\text{sum over the orbits of elements of the type} \quad \frac{1}{2}(\pm 3\varepsilon_1 \pm \varepsilon_2 \pm \varepsilon_3 \pm \varepsilon_4),$$
$$\frac{1}{2}(\pm \varepsilon_1 \pm 3\varepsilon_2 \pm \varepsilon_3 \pm \varepsilon_4), \quad \frac{1}{2}(\pm \varepsilon_1 \pm \varepsilon_2 \pm 3\varepsilon_3 \pm \varepsilon_4), \quad \frac{1}{2}(\pm \varepsilon_1 \pm \varepsilon_2 \pm \varepsilon_3 \pm 3\varepsilon_4).$$

\begin{remark}
Following the tradition, $Q$ corresponds to the orbits of elements belonging to the lattice and $P$ corresponds to the orbits of the weight vectors. 
The subscripts of $P$ and $Q$ are equal to the length of each element of the orbit. We use different subscripts for polynomials $P_1$ and $P_{\frac{4}{4}}$ which represent different orbits. This separation is correct because in the general case of the lattice $D_n$ the dual lattice $D_n^{\vee}$ contains the vector $\frac{1}{2}(\varepsilon_1+\ldots+\varepsilon_n)$ of the length $\frac{n}{4}$. The notations $P_3$ and $P_{\frac{12}{4}}$ are similar. In the case of $Q_4^1$ and $Q_4^2$ this way to distinguish the orbits does not work, and we use superscripts instead.
\end{remark}

One can prove by direct computation the following representations of generators of the algebra of weak $W(D_4)$-invariant Jacobi forms, using their explicit construction from Theorem \ref{TheoremD4}:
$$\varphi_{0,1} = 32Q_0+P_{\frac{4}{4}}+q\cdot(\ldots),$$
$$\varphi_{-2,1} = 24Q_0-P_1-P_{\frac{4}{4}}+q\cdot(\ldots),$$
$$\varphi_{-4,1} = -2P_1+P_{\frac{4}{4}}+q\cdot(\ldots),$$
$$\omega_{-4,1} = P^+_{\frac{4}{4}}-P^-_{\frac{4}{4}}+q\cdot(\ldots),$$
$$\varphi_{-6,2} = -320Q_0+112P_1 - 32Q_2+4P_3+4Q_4^1+q\cdot(\ldots).$$

Also, it follows from the Example \ref{ThetaD4} that 
$$\varphi^{(2)}_{-4,1} = -P_1+P_{\frac{4}{4}}^+ + q\cdot(\ldots);$$
$$\varphi^{(3)}_{-4,1} = P_1-P_{\frac{4}{4}}^- + q\cdot(\ldots).$$

\begin{remark}
The polynomials $Q_4^2$ and $P_{\frac{12}{4}}$ defined above are not contained in $q^0$-coefficients of $W(D_4)$-invariant Jacobi forms, but they appear below in the construction of $O(D_4)$-invariant generators.
\end{remark}

\section{Construction of generators}

We mark weak Jacobi forms for the lattice  $F_4(2)$ (that is, $O(D_4)$-invariant forms) with the superscript $F_4$. However, to shorten notation we omit this label for $W(D_4)$-invariant Jacobi forms. The main result of the paper is the following theorem.

\begin{theorem}\label{MainTheorem}
The ring of weak Jacobi forms of even index for the lattice $F_4$, which are invariant under the action of the Weyl group $W(F_4)$, has a structure of a polynomial algebra over the ring of modular forms. Namely,
$$J_{*, 2*}^{w, W}(F_4)\simeq J_{*, *}^{w, W}(F_4(2))=M_*[\varphi^{F_4}_{0,1}, \varphi^{F_4}_{-2,1}, \varphi^{F_4}_{-6,2}, \varphi^{F_4}_{-8,2}, \varphi^{F_4}_{-12,3}]\simeq J_{*, *}^{w, O}(D_4).$$
\end{theorem}

\begin{remark}
A first idea of the proof is to construct the generators of the algebra of weak $W(F_4)$-invariant Jacobi forms as the averages of the generators of $J_{*,*}^{w,W}(D_4)$ over the group $W(F_4)/W(D_4) \simeq S_3$. However, a straightforward computation proves that the average of the form $\varphi_{-4,1}$ is identically equal to zero, and this is a stumbling block in the construction, because this form is central in the case of the lattice $D_4$. Nevertheless, we construct generators of $J_{*,*}^{w,W}(F_4)$ by use of generators of $J_{*,*}^{w,W}(D_4)$, but this construction is a little bit tricky than averaging over $S_3$. 
 \end{remark}

The Weyl group $W(F_4)$ is generated by reflections in $\pm \varepsilon_i\pm \varepsilon_j$, $\pm \varepsilon_i$ and $\frac{1}{2}(\pm \varepsilon_1 \pm \varepsilon_2 \pm \varepsilon_3 \pm \varepsilon_4)$, while the Weyl group $W(D_4)$ is generated by reflections only in $\pm \varepsilon_i\pm \varepsilon_j$. 
Suppose that we have constructed generators in the case of the lattice $F_4$. Let us choose  one representative of each type of vectors $\pm \varepsilon_i$ (this corresponds to the sign changes of any number of $\mathfrak{z}$-coordinates), $\frac{1}{2}(\pm \varepsilon_1 \pm \varepsilon_2 \pm \varepsilon_3 \pm \varepsilon_4)$ with odd number of $+$, $\frac{1}{2}(\pm \varepsilon_1 \pm \varepsilon_2 \pm \varepsilon_3 \pm \varepsilon_4)$ with even number of $+$, and call reflections in them by $f$, $g$, $h$, respectively.
Because of $W(D_4)$-invariance of generators for $D_4$
it is enough to check that constructed generators for $F_4$ are invariant under these reflections.

\begin{statement}\label{Orbits}
The reflections $f$, $g$ and $h$ act on polynomials $P_1$, $P_{\frac{4}{4}}^+$ and $P_{\frac{4}{4}}^-$ as follows:
$$
\begin{aligned}
f: \quad & P_{\frac{4}{4}}^{-} \mapsto P_{\frac{4}{4}}^{+}, \quad P_{\frac{4}{4}}^{+} \mapsto P_{\frac{4}{4}}^{-} \quad P_1 \mapsto P_1;
\end{aligned}
$$
$$
\begin{aligned}
g: \quad & P_{\frac{4}{4}}^- \mapsto P_1, \quad P_{\frac{4}{4}}^+ \mapsto P_{\frac{4}{4}}^+, \quad P_1 \mapsto P_{\frac{4}{4}}^-;
\end{aligned}
$$
$$
\begin{aligned}
h: \quad & P_{\frac{4}{4}}^- \mapsto P_{\frac{4}{4}}^-, \quad P_{\frac{4}{4}}^+ \mapsto P_1, \quad P_1 \mapsto P_{\frac{4}{4}}^+.
\end{aligned}
$$
\end{statement}
\begin{proof}
Direct calculation.
\end{proof}

\begin{statement}\label{lemF_4}
The reflections $f$, $g$ and $h$ act on Jacobi forms $\varphi_{-4,1}^{(2)}$ and $\varphi_{-4,1}^{(3)}$ as follows:
$$ f: \varphi_{-4,1}^{(2)}\mapsto -\varphi_{-4,1}^{(3)}, \quad \varphi_{-4,1}^{(3)} \mapsto  -\varphi_{-4,1}^{(2)};$$
$$ g: \varphi_{-4,1}^{(2)}\mapsto \varphi_{-4,1}^{(2)}+\varphi_{-4,1}^{(3)}, \quad \varphi_{-4,1}^{(3)} \mapsto -\varphi_{-4,1}^{(3)};$$
$$ h: \varphi_{-4,1}^{(2)}\mapsto -\varphi_{-4,1}^{(2)}, \quad \varphi_{-4,1}^{(3)} \mapsto \varphi_{-4,1}^{(2)}+\varphi_{-4,1}^{(3)}.$$
\end{statement}
\begin{proof} Like in the previous statement, the proof amounts to a direct calculation due to the explicit expressions of the Jacobi forms $\varphi_{-4,1}^{(2)}$ and $\varphi_{-4,1}^{(3)}$ in terms of $\vartheta^{(2)}$ and $\vartheta^{(3)}$ given in Example \ref{ThetaD4}. 
\end{proof}

\subsection{Construction of $\varphi_{-12, 3}^{F_4}$ and $\varphi_{-8, 2}^{F_4}$}
\label{subsec:Chevalley}

As it was mentioned in Definition \ref{sysF4}, the Weyl group $W(F_4)$ is the semidirect product of $W(D_4)$ and $S_3$. 
Realization of the involutions $f$, $g$ and $h$ as transpositions in $S_3$ defines the following matrix representation of $S_3$:
$$f=\left(\begin{matrix}
0 &-1\\
-1 & 0
\end{matrix}\right), \;
g=\left(\begin{matrix}
1 & 1\\
0 & -1
\end{matrix}\right),\;
h=\left(\begin{matrix}
-1 & 0\\
1 & 1
\end{matrix}\right).$$
Hence, by the Chevalley theorem over $\mathbb{C}$ (see e.g. \cite[V, \S 5]{Bur}), the subalgebra of invariants of algebra of polynomials in  $\varphi_{-4,1}^{(2)}$ and $\varphi_{-4,1}^{(3)}$ is freely generated by two polynomials of degrees $2$ and $3$, respectively. These polynomials are weak $W(F_4)$-invariant Jacobi forms $\varphi_{-12, 3}^{F_4}$ and $\varphi_{-8, 2}^{F_4}$ automatically. Let us find them.

The form $\varphi_{-12, 3}^{F_4}$ is equal to 
$$P(\varphi_{-4,1}^{(2)}, \varphi_{-4,1}^{(3)}) = a(\varphi_{-4,1}^{(2)})^3+b(\varphi_{-4,1}^{(2)})^2\varphi_{-4,1}^{(3)}+c\varphi_{-4,1}^{(2)}(\varphi_{-4,1}^{(3)})^2+d(\varphi_{-4,1}^{(3)})^3$$
with some complex coefficients, due to the weight and index considerations. A direct calculation of $S_3$ action shows that up to a constant factor there is only one such invariant Jacobi form:
$$2\varphi_{-12, 3}^{F_4} = 2(\varphi_{-4,1}^{(2)})^3+3(\varphi_{-4,1}^{(2)})^2\varphi_{-4,1}^{(3)}-3\varphi_{-4,1}^{(2)}(\varphi_{-4,1}^{(3)})^2-2(\varphi_{-4,1}^{(3)})^3.$$
Using the fact that $\varphi_{-4,1}=\varphi_{-4,1}^{(2)} - \varphi_{-4,1}^{(3)}$ and $\omega_{-4,1} = \varphi_{-4,1}^{(2)}+\varphi_{-4,1}^{(3)},$ we obtain
$$8\varphi_{-12, 3}^{F_4} = 9 \varphi_{-4,1}\omega_{-4,1}^2  - \varphi_{-4,1}^3 = \varphi_{-4,1}(9\omega_{-4,1}^2 - \varphi_{-4,1}^2).$$
As for $\varphi_{-8, 2}^{F_4}$, it can be written as
$$Q(\varphi_{-4,1}^{(2)}, \varphi_{-4,1}^{(3)}) = a (\varphi_{-4,1}^{(2)})^2+b\varphi_{-4,1}^{(2)}\varphi_{-4,1}^{(3)}+c(\varphi_{-4,1}^{(3)})^2$$ 
with complex coefficients. Again, a direct calculation proves that up to a constant factor there is only one invariant form
$$\varphi_{-8, 2}^{F_4} = (\varphi_{-4,1}^{(2)})^2+\varphi_{-4,1}^{(2)}\varphi_{-4,1}^{(3)}+(\varphi_{-4,1}^{(3)})^2.$$
Using the fact that $\varphi_{-4,1}=\varphi_{-4,1}^{(2)} - \varphi_{-4,1}^{(3)}$ and $\omega_{-4,1} = \varphi_{-4,1}^{(2)}+\varphi_{-4,1}^{(3)},$ we get
$$4\varphi_{-8, 2}^{F_4} =3\omega_{-4,1}^2+\varphi_{-4,1}^2.$$

\subsection{Products of some invariant polynomials}\label{subsec:products}

For the further computations we need to find the $q^0$-terms of all $W(D_4)$-invariant forms of index 2. The $q^0$-term of $\varphi_{-6,2}$ is known and the $q^0$-terms of generators of index 1 are linear combinations of $Q_0$, $P_1$, $P_{\frac{4}{4}}^{+}$ and $P_{\frac{4}{4}}^{-}$.
Hence, we need to calculate the pairwise products of these four polynomials, with the exception that we do not need to know $P_1 \cdot P_{\frac{4}{4}}^{+}$ and $P_1 \cdot P_{\frac{4}{4}}^{-}$ separately, only their sum $P_1 \cdot P_{\frac{4}{4}}$ is required. Obviously, the multiplication by $Q_0$ is trivial. The other products are the following:
$$(P_{\frac{4}{4}}^{+})^2=8Q_0+2Q_2+Q_4^{1,+},$$
$$(P_{\frac{4}{4}}^{-})^2=8Q_0+2Q_2+Q_4^{1,-},$$
where $Q_4^{1,-}$ and $Q_4^{1,+}$ correspond to the summands with odd and even number of $+$ signs, respectively, 
$$P_{\frac{4}{4}}^{-}\cdot P_{\frac{4}{4}}^{+}=4P_1+P_3,$$
$$(P_1)^2=8Q_0+2Q_2+Q_4^2,$$
$$P_1 \cdot P_{\frac{4}{4}}=4P_{\frac{4}{4}}+P_{\frac{12}{4}}.$$

Using all these formulas we obtain all pairwise products of generators of index 1 for $D_4$.

\subsection{Construction of $\varphi_{-6, 2}^{F_4}$}

In the case of the lattice $F_4$ the differential operator $H_k$ acts on the Jacobi form $\varphi_{k, m}^{F_4}=\sum_{n=0}^{\infty} \sum_{l \in F_4} a(n, l) q^n \zeta^l$ by
$$
H_{k}(\varphi_{k, m}^{F_4}) = \sum_{n=0}^{\infty} \sum_{l \in F_4}\left(n - \frac{1}{2m}(l,l)\right) a(n, l) q^n \zeta^l +(2k-4)G_2\varphi_{k, m}^{F_4}.
$$
Let us apply the differential operator $H_{-8}$ to the constructed Jacobi form $\varphi_{-8, 2}^{F_4}$. By use of the above pairwise products  we obtain
$$\varphi_{-8, 2}^{F_4}=\frac{3\omega_{-4,1}^2+\varphi_{-4,1}^2}{4}=24Q_0-4P_1-4P_{\frac{4}{4}}+6Q_2-P_3-P_{\frac{12}{4}}+Q_4+q\cdot(\ldots).$$
Hence, as one can calculate,
$$12H_{-8}(\varphi_{-8, 2}^{F_4}) = 240Q_0-28P_1-28P_{\frac{4}{4}}+24Q_2-P_3-P_{\frac{12}{4}}-2Q_4+q\cdot(\ldots) \neq 0,$$
and we may choose this form as $\varphi_{-6, 2}^{F_4}$. 
Let us express this form in terms of the generators of $J_{*,*}^{w,W}(D_4)$. As we know,
$$J_{-6,2}^{w, W}(D_4) = \langle \varphi_{-6,2}; \varphi_{-2,1}\varphi_{-4,1}; \varphi_{-2,1}\omega_{-4,1} \rangle.$$
However, the form $\varphi_{-2,1}\omega_{-4,1}$ is anti-invariant under the odd number of sign changes of $\mathfrak{z}$-coordinates. Thus, 
$$ \varphi_{-6, 2}^{F_4} = a\varphi_{-6,2} + b\varphi_{-2,1}\varphi_{-4,1}.$$
We know that
$$\varphi_{-6,2} = -320Q_0+112P_1 - 32Q_2+4P_3+4Q_4^1+q\cdot(\ldots);$$
$$\varphi_{-2,1}\varphi_{-4,1} =(24Q_0-P_1-P_{\frac{4}{4}})(-2P_1+P_{\frac{4}{4}})+q\cdot(\ldots)=$$
$$=-56P_1+28P_{\frac{4}{4}}-2P_3+P_{\frac{12}{4}}-Q_4^1+2Q_4^2.$$
Therefore, 
$$
240Q_0-28P_1-28P_{\frac{4}{4}}+24Q_2-P_3-P_{\frac{12}{4}}-2Q_4=$$
$$
a( -320Q_0+112P_1 - 32Q_2+4P_3+4Q_4^1)-b(56P_1-28P_{\frac{4}{4}}+2P_3-P_{\frac{12}{4}}+Q_4^1-2Q_4^2),$$
and $a = -\frac{3}{4}$, $b= -1$. As a result, we obtain
$$-4\varphi_{-6, 2}^{F_4} = 3\varphi_{-6,2} + 4\varphi_{-2,1}\varphi_{-4,1}.$$

\subsection{Construction of $\varphi_{-2, 1}^{F_4}$}

Similar to the previous case, let us apply the differential operator $H_{-6}$ to $\varphi_{-6, 2}^{F_4}$. We get
$$H_{-6}(\varphi_{-6, 2}^{F_4})=$$$$= \frac{1}{12}\left(1920Q_0-140P_1-140P_{\frac{4}{4}}+48Q_2+P_3+P_{\frac{12}{4}}+8Q_4+q\cdot(\ldots)\right).$$
However, in this case 
$$J_{-4,2}^{w, W}(D_4) = \langle E_4\varphi_{-4,1}^2; E_4\omega_{-4,1}^2; \varphi_{-2,1}^2; \varphi_{0,1}\varphi_{-4,1}; \varphi_{0,1}\omega_{-4,1}\rangle.$$
The latter Jacobi form is not invariant  under the odd number of sign changes of $\mathfrak{z}$-coordinates. So, it does not appear in the expression of $H_{-6}(\varphi_{-6, 2}^{F_4})$ as a linear combination of Jacobi forms of weight $-4$ and index 2 for $D_4$. Let us calculate $q^0$-term of other forms. By use of the pairwise products from Section \ref{subsec:products}, we obtain
$$E_4\omega_{-4,1}^2 = 16Q_0-8P_1+4Q_2-2P_3+Q_4^1;$$ 
$$ E_4\varphi_{-4,1}^2 = 48Q_0+8P_1-16P_{\frac{4}{4}}+12Q_2+2P_3-4P_{\frac{12}{4}}+Q_4^1+4Q_4^2;$$
$$\varphi_{-2,1}^2 = 600Q_0-40P_1-40P_{\frac{4}{4}}+6Q_2+2P_3+2P_{\frac{12}{3}}+Q_4;$$
$$ \varphi_{-4,1}\varphi_{0,1} = 16Q_0-56P_1-40P_{\frac{4}{4}}+4Q_2+2P_3-2P_{\frac{12}{4}}+Q_4^1.$$
Suppose that
$H_{-6}(\varphi_{-6, 2}^{F_4})=aE_4\omega_{-4,1}^2+b E_4\varphi_{-4,1}^2+c\varphi_{-2,1}^2+d\varphi_{0,1}\varphi_{-4,1}.$
Then we get a system of linear equations, which has the unique solution $a = \frac{15}{4}$, $b= \frac{3}{4}$, $c=3$, $d=0$. Hence, 
$$H_{-6}(\varphi_{-6, 2}^{F_4})= \frac{15}{4}E_4\omega_{-4,1}^2+\frac{3}{4} E_4\varphi_{-4,1}^2+3\varphi_{-2,1}^2=$$
$$=\frac{5}{4}(3E_4\omega_{-4,1}^2+ E_4\varphi_{-4,1}^2+3\varphi_{-2,1}^2).$$
By subtracting 
$$H_{-6}(\varphi_{-6, 2}^{F_4})-5E_4\varphi_{-8,2}^{F_4}=H_{-6}(\varphi_{-6, 2}^{F_4})-\frac{5}{4}(3E_4\omega_{-4,1}^2+E_4\varphi_{-4,1}^2),$$
we obtain that $0\neq\varphi_{-2,1}^2 \in J_{-4,2}^{w, W}(F_4)$. 

Let us prove now that $\varphi_{-2,1}$ is $W(F_4)$-invariant. By construction it is invariant under the action of $W(D_4)$ and any number of sign changes of $\mathfrak{z}$-coordinates. So, we only need to check invariance under reflections in roots of type $\frac{1}{2}(\pm \varepsilon_1 \pm \varepsilon_2\pm \varepsilon_3 \pm \varepsilon_4)$.
Because of invariance of $\varphi_{-2,1}^2$ under this transformations, the form $\varphi_{-2,1}$ is invariant or anti-invariant. However, the $q^0$-term of $\varphi_{-2,1}$ is equal to $24Q_0-P_1-P_{\frac{4}{4}}$, and it is $W(F_4)$-invariant by Statement \ref{Orbits}.
Therefore $\varphi_{-2,1}$ is $W(F_4)$-invariant.

\subsection{Construction of $\varphi_{0, 1}^{F_4}$}

Finally, the form of weight 0 and index 1 can be obtained as 
$$\varphi_{0,1}^{F_4}=6H_{-2}(\varphi_{-2,1}^{F_4}) = 48Q_0+P_1+P_{\frac{4}{4}} +q\cdot(\ldots).$$
As we know, any Jacobi form of weight 0 and index 1 which is invariant under any number of sign changes of $\mathfrak{z}$-coordinates is a linear combination of $\varphi_{0,1}$ and $E_4\varphi_{-4,1}$. 
Therefore, by comparing the coefficients we obtain the representation
$$2\varphi_{0,1}^{F_4}=3\varphi_{0,1}-E_4\varphi_{-4,1}.$$

\section{Algebraic independence and sufficiency of the constructed Jacobi forms}

Now let us prove that the constructed Jacobi forms for the lattice $F_4$ are indeed generators of the algebra of weak $W(F_4)$-invariant Jacobi forms. 

\begin{lemma}\label{AlgIndep}
Jacobi forms $\varphi_{0,1}^{F_4}$, $\varphi_{-2,1}^{F_4}$, $\varphi_{-6,2}^{F_4}$, $\varphi_{-8,2}^{F_4}$ and $\varphi_{-12,3}^{F_4}$ are algebraically independent over the ring of modular forms.
\end{lemma}
\begin{proof}
Let us recall the expressions of these forms in terms of generators of the algebra of weak $W(D_4)$-invariant Jacobi forms: 
$$2\varphi_{0,1}^{F_4} = 3\varphi_{0,1}-E_4\varphi_{-4,1},$$
$$\varphi_{-2,1}^{F_4} = \varphi_{-2,1},$$
$$-4\varphi_{-6,2}^{F_4} = 3\varphi_{-6,2} + 4\varphi_{-2,1}\varphi_{-4,1},$$
$$4\varphi_{-8,2}^{F_4} = 3\omega_{-4,1}^2+\varphi_{-4,1}^2,$$
$$8\varphi_{-12,3}^{F_4} = \varphi_{-4,1}(9\omega_{-4,1}^2 - \varphi_{-4,1}^2).$$
Suppose that these forms are algebraically dependent. Then
$$U(\varphi_{0,1}^{F_4}, \varphi_{-2,1}^{F_4}, \varphi_{-6,2}^{F_4}, \varphi_{-8,2}^{F_4}, \varphi_{-12,3}^{F_4})$$
for some polynomial $U$ over the ring of modular forms. Consider its monomials containing the maximal degree $k$ of the form $\varphi_{0,1}^{F_4}$. Their sum is equal to
$$(\varphi_{0,1}^{F_4})^kU_1(\varphi_{-2,1}^{F_4}, \varphi_{-6,2}^{F_4}, \varphi_{-8,2}^{F_4}, \varphi_{-12,3}^{F_4})$$
for some polynomial $U_1$.
Now let us set $\varphi_{0,1}^{F_4} = \frac{3}{2}\varphi_{0,1}-\frac{1}{2}E_4\varphi_{-4,1}$. Then the polynomial $U\equiv 0$ can be written as
$$\varphi_{0,1}^kU_1(\varphi_{-2,1}^{F_4}, \varphi_{-6,2}^{F_4}, \varphi_{-8,2}^{F_4}, \varphi_{-12,3}^{F_4})+\varphi_{0,1}^{k-1}U_2(\varphi_{-2,1}^{F_4}, \varphi_{-6,2}^{F_4}, \varphi_{-8,2}^{F_4}, \varphi_{-12,3}^{F_4})+\ldots.$$
Considering this relation as relation for $D_4$ we obtain that $U_1= 0$. Hence, $U$ does not depend on $\varphi_{0,1}^{F_4}$, because of maximality of $k$.

The same argument shows that $U$ does not depend on $\varphi_{-6,2}^{F_4}$, because the Jacobi form $\varphi_{-6,2}$ does not appear in expression of other forms in terms of generators for $D_4$. Then the same argument shows that $U$ does not depend on $\varphi_{-2,1}^{F_4}$.

Thus we need only to check that $\varphi_{-8,2}^{F_4}$ and $\varphi_{-12,3}^{F_4}$ are algebraically independent over the ring of modular forms. Suppose the contrary. Then $U(\varphi_{-8,2}^{F_4}, \varphi_{-12,3}^{F_4})\equiv 0$. Consider the explicit formulas for these forms in terms of $\varphi_{-4,1}$ and $\omega_{-4,1}$, which were obtained in \hyperref[subsec:Chevalley]{\S 4.1}. Setting $z_4=0$ we get the non-trivial relation on $\varphi_{-4,1}^{D_3} = \varphi_{-4,1}\bigg|_{z_4=0}$ for the lattice $D_3$, because $\omega_{-4,1}\bigg|_{z_4=0}=0$. And this is a contradiction to Theorem \ref{TheoremD4}.
\end{proof}

\begin{lemma}\label{Generate}
Any weak $W(D_4)$-invariant Jacobi forms is a polynomial in $\varphi_{0,1}^{F_4}$, $\varphi_{-2,1}^{F_4}$, $\varphi_{-6,2}^{F_4}$, $\varphi_{-8,2}^{F_4}$ and $\varphi_{-12,3}^{F_4}$ over the ring of modular forms. 
\end{lemma}
\begin{proof}
Consider an arbitrary Jacobi form $\Phi_{k, m} \in J_{*,*}^{w, W}(F_4)$. As we know, this form can be expressed as a polynomial in  generators of $J_{*,*}^{w, W}(D_4)$. Thus, for some polynomial $U$ over $\mathbb{C}$
$$\Phi_{k, m} = U(E_4, E_6, \varphi_{0,1}, \varphi_{-2,1}, \varphi_{-4,1}, \varphi_{-6,2}, \omega_{-4,1}).$$
Let us substitute into $U$ the expressions 
$$\varphi_{0,1} = \frac{2\varphi_{0,1}^{F_4}+E_4\varphi_{-4,1}}{3},$$
$$\varphi_{-6,2} = \frac{\varphi_{-6,2}^{F_4}-4\varphi_{-2,1}\varphi_{-4,1}}{3}.$$
Then we obtain that $\Phi_{k, m}$ equals 
$$\sum_{\alpha=(n_1, n_2, n_3, n_4, n_5)} a_{\alpha} U_{\alpha}(\varphi_{-4,1}, \omega_{-4,1}) E_4^{n_1}E_6^{n_2}(\varphi_{0,1}^{F_4})^{n_3}(\varphi_{-2,1}^{F_4})^{n_4}(\varphi_{-6,2}^{F_4})^{n_5},$$
where $\alpha$ is a multi-index, $a_{\alpha}$ are complex numbers, $U_{\alpha}$ is a polynomial in $\varphi_{-4,1}$ and $\omega_{-4,1}$ over $\mathbb{C}$. This representation is correct since $\varphi_{-2,1}^{F_4}=\varphi_{-2,1}$.
By construction, each form
$$E_4^{n_1}E_6^{n_2}(\varphi_{0,1}^{F_4})^{n_3}(\varphi_{-2,1}^{F_4})^{n_4}(\varphi_{-6,2}^{F_4})^{n_5}$$
is invariant under the action of $W(F_4)$. Hence, all $U_{\alpha}$ are also invariant. Then each of them is a polynomial in $\varphi_{-8,2}^{F_4}$ and $\varphi_{-12,3}^{F_4}$, as it has been noticed in \hyperref[subsec:Chevalley]{\S 4.1}.
\end{proof}

Lemma \ref{AlgIndep} and Lemma \ref{Generate} together complete the proof of Theorem \ref{MainTheorem}.

D. Adler \par
\smallskip
International laboratory of mirror symmetry and automorphic forms
\par
NRU HSE, Moscow
\par
\smallskip
{\tt dmitry.v.adler@gmail.com}
\vskip0.5cm


\begin{thebibliography}{}

\bibitem[1]{Ch} C. Chevalley, 
\textit{Invariants of finite groups generated by reflections.}
Amer. J. Math. \textbf{77} (1955), 778--782

\bibitem[2]{BS} J.N. Bernstein, O.V. Schwarzman,
\textit{Chevalley theorem for complex crystallographic Coxeter groups.}
Functional Anal. Appl. \textbf{12} (1978).

\bibitem[3]{Lo1} E. Looijenga, {\it Root Systems and Elliptic Curves.} 
Inv. Mathem. {\bf 38} (1976), 17--32.

\bibitem[4]{Lo2}E. Looijenga, {\it Invariant Theory for Generalized Root 
Systems.} Inv. Mathem. {\bf 61} (1980), 1--32.

\bibitem[5]{KP} V. Kac, D. Peterson, 
\textit{Infinite-dimensional Lie algebras, theta functions and modular forms.}
Adv. in Math., \textbf{53} (1984), 125--264.

\bibitem[6]{Wirt} K. Wirthm\"uller, \textit{Root systems and Jacobi forms.} 
Comp. Math. {\bf 82} (1992), 293--354.

\bibitem[7]{Wa} H. Wang, \textit{Weyl invariant $E_8$ Jacobi forms.} 
\url{arXiv:1801.08462}

\bibitem[8]{Sa2} K. Saito, {\it Extended Affine Root Systems I (Coxeter 
transformations).} Publ. RIMS, {\bf  21} (1985), 75--179.

\bibitem[9]{Sa3} K. Saito, {\it Extended Affine Root Systems II (Flat 
Invariants).} Publ. RIMS,  {\bf 26} (1990), 15--78.

\bibitem[10]{Du} B.A. Dubrovin, 
\textit{Geometry of 2D topological field theories in Integrable Systems and
Quantum Groups.}
Montecatini, Terme 1993, ed. Francaviglia, M. and Greco, S..
Springer lecture notes in mathematics, \textbf{1620}, Springer-Verlag 1996, 120--348.

\bibitem[11]{Sat} I. Satake
{\it Flat Structure for the Simple Elliptic Singularity
of Type  ${\widetilde E}_6$ and Jacobi Form.}
Proc. Japan Acad., {\bf 69}, Ser. A (1993) No. 7, 247--251.

\bibitem[12]{Ber1} M. Bertola, 
{\it Frobenius manifold structure on orbit space of Jacobi group; Part I.} 
Differential Geom. Appl. {\bf13} (2000), 19--41. 

\bibitem[13]{Ber2} M. Bertola, {\it Frobenius manifold structure on orbit space of Jacobi group; Part II.} 
Differential Geom. Appl. {\bf 13 (3)} (2000), 213--233.

\bibitem[14]{AG} D. Adler, V. Gritsenko, {\it The $D_8$-tower of weak Jacobi forms and applications.}
J. Geom. Phys., electronically published on February 6, 2020, DOI: https://doi.org/10.1016/j.geomphys.2020.103616 (to appear in print).

\bibitem[15]{Bur} N. Bourbaki, 
{\it Groupes et Alg\`ebres de Lie, Ch. 4, 5, 6.} 
Masson, 1981.

\bibitem[16]{EZ} M. Eichler, D. Zagier, 
{\it The theory of Jacobi forms.} Progress in Mathematics {\bf 55}.
Birkh\"auser, Boston, Mass. (1985).

\bibitem[17]{M} D.~Mumford, {\it Tata lectures on theta I.}
Progress in Mathem. {\bf 28}, Birkh\"auser, Boston, Mass., 1983.

\bibitem[18]{GrJ} V.A. Gritsenko, 
\textit{Jacobi modular forms: 30~ans apr\`es.}
Course of lectures on Coursera 2016--2018.
\noindent
\url{https://ru.coursera.org/learn/modular-forms-jacobi}

\bibitem[19]{CG} F. Cl\'ery,   V. Gritsenko,
\textit{Modular forms of orthogonal type and Jacobi theta-series.}
Abh. Math. Semin. Univ. Hambg \textbf{83} (2013), 187--217.

\end{thebibliography}
\end{document}